\documentclass{article}
\usepackage{amsmath,amsfonts,amsthm,amssymb}
\usepackage[top=2cm,bottom=2cm,left=2cm,right=2cm]{geometry}
\usepackage{graphicx,hyperref}

\theoremstyle{plain}
\newtheorem{theorem}{Theorem}
\newtheorem{corollary}[theorem]{Corollary}
\newtheorem{lemma}{Lemma}
\newtheorem{conjecture}{Conjecture}

\theoremstyle{definition}

\newtheorem{example}{Example}

\newcommand{\ints}{\mathbb{Z}}

\title{Minimizing the regularity of maximal regular antichains of $2$- and $3$-sets}
\author{
Thomas Kalinowski \\ 
\textit{University of Newcastle, Australia}
\and
Uwe Leck\\
\textit{University of Wisconsin-Superior, U.S.A.}
\and
Christian Reiher\\
\textit{University of Hamburg, Germany}
\and
Ian T. Roberts\\
\textit{Charles Darwin University, Australia}
}

\begin{document}

\maketitle

\begin{abstract}
Let $n\geqslant 3$ be a natural number. We study the problem to find the smallest $r$ such that there is a family $\mathcal{A}$ of $2$-subsets and $3$-subsets of $[n]=\{1,2,\ldots,n\}$ with the following properties: (1) $\mathcal{A}$ is an antichain, i.e., no member of $\mathcal A$ is a subset of any other member of $\mathcal A$, (2) $\mathcal A$ is maximal, i.e., for every $X\in 2^{[n]}\setminus\mathcal A$ there is an $A\in\mathcal A$ with $X\subseteq A$ or $A\subseteq X$, and (3) $\mathcal A$ is $r$-regular, i.e., every point $x\in[n]$ is contained in exactly $r$ members of $\mathcal A$. We prove lower bounds on $r$, and we describe constructions for regular maximal antichains with small regularity. 
\end{abstract}

\section{Introduction}\label{sec:intro}
Let $n\geqslant 3$ be an integer and $[n]:=\{1,2,\dots,n\}$. By $2^{[n]}$ we denote the family of all subsets of $[n]$ and by $\binom{[n]}k$ the family of all $k$-subsets of $[n]$. A family $\mathcal{A}\subseteq 2^{[n]}$ is an \emph{antichain} if $A\not\subseteq B$ for all distinct $A,B\in\mathcal{A}$. An antichain $\mathcal A$ is called \emph{flat} if it is contained in two consecutive levels of $2^{[n]}$, i.e., if $\mathcal A\subseteq\tbinom{[n]}{k}\cup\tbinom{[n]}{k+1}$ for some $k$. The remarkable \emph{Flat Antichain Theorem}, which follows from results of Lieby~\cite{lieby1999extremal} (see also~\cite{lieby2004antichains}) and Kisv{\"o}lcsey~\cite{kisvolcsey2006flattening}, says that for every antichain $\mathcal A$ there is an equivalent flat antichain $\mathcal A'$, where the equivalence relation is defined by: $\mathcal A'$ is equivalent to $\mathcal A$  if and only if $\lvert\mathcal A'\rvert=\lvert\mathcal A\rvert$ and $\sum_{A\in\mathcal A'}\lvert A\rvert=\sum_{A\in\mathcal A}\lvert A\rvert$. Griggs et al.~\cite{GriggsHartmannLeckRoberts2012} showed that the flat antichains minimize the \emph{BLYM-values} $\sum_{A\in\mathcal A}\tbinom{n}{\lvert A\rvert}^{-1}$ within their equivalence classes, and more generally, they minimize (maximize) $\sum_{A\in\mathcal A}w(\lvert A\rvert)$ for every convex (concave) function $w$. Gr\"uttm\"uller et al. \cite{GruettmuellerHartmannKalinowskiLeckRoberts2009} initiated the study of maximal flat antichains. In particular, they asked for the minimum size of a maximal antichain consisting only of $k$-sets and $(k+1)$-sets, and settled the first nontrivial case $k=2$. Gerbner et al.~\cite{gerbner2011saturating} derived asymptotic bounds for general $k$ and studied some related problems for $r$-Sperner families, i.e., families of subsets of $[n]$ which do not contain $r+1$ sets that form a chain with respect to inclusion. A different generalization was proposed in~\cite{kalinowski2012mfac} where the flatness condition is replaced by the requirement that the cardinalities of the members of the set family $\mathcal A$ have to be in a fixed small set $K$. B\"ohm~\cite{boehm2011PhD,bohm2012k,boehm2012MaximalFlatRegularAntichains} proposed the study of regular antichains, i.e., antichains $\mathcal A$ such that every element of the ground set appears in the same number of sets in $\mathcal A$. In particular, he studied the problem for which values of $r$ there are maximal flat antichains such that every element of the ground set occurs exactly $r$ times. The same problem without the maximality and flatness conditions has been studied by Roberts and B\"ohm~\cite{roberts2013existence}.  

In this paper we focus on maximal flat antichains consisting of $2$-sets and $3$-sets: we consider \emph{maximal $(2,3)$-antichains} $\mathcal A$ on $[n]:=\{1,2,\ldots,n\}$, i.e., $\mathcal A=\mathcal A_2\cup\mathcal A_3$ with $\mathcal A_2\subseteq\tbinom{[n]}{2}$ and $\mathcal A_3\subseteq\tbinom{[n]}{3}$. The antichain $\mathcal A$ is called $r$-regular if every element of $[n]$ is contained in exactly $r$ members of $\mathcal A$. Clearly, the maximum possible value of $r$ is $\tbinom{n-1}{2}$, obtained by choosing $\mathcal A_3=\tbinom{[n]}{3}$, but the smaller value $r=n-1$ can be achieved by choosing $\mathcal A_2 = \tbinom{[n]}{2}$.  We are interested in antichains of small regularity, so henceforth we will assume that $r<n-1$.

Following \cite{GruettmuellerHartmannKalinowskiLeckRoberts2009,kalinowski2012mfac} we associate a graph $G$ with $\mathcal A$. The vertex set is $[n]$ and the edge set is $\tbinom{[n]}{2}\setminus\mathcal A_2$. The elements of $\mathcal A_3$ are the triangles in $G$, and every edge of $G$ is contained in at least one triangle. For a vertex $x$ let $d(x)$ denote its degree and $t(x)$ the number of triangles containing $x$, i.e., the number of edges in the neighbourhood of $x$. Similarly, for an edge $xy$, $t(xy)$ is the number of triangles containing $xy$, i.e., the number of common neighbours of $x$ and $y$. The antichain $\mathcal A$ is $r$-regular if and only if in the corresponding graph, for all vertices $x$,
\[d(x)-t(x)=\overline r\]
where $\overline r$ is defined as $\overline r=n-1-r>0$. In the present paper we study the question for which values of $r$ an $r$-regular maximal $(2,3)$-antichain exists. 

In Section~\ref{sec:lower_bounds} we prove two lower bounds for $r$. The first one $r\geqslant (5n-9)/6$ is valid for all $n$ (Corollary~\ref{cor:56bound}) and the second one $r\geqslant(9/10+o(1))n$ holds asymptotically for $n\to\infty$ (Corollary~\ref{cor:asump_9_10}). Section~\ref{sec:constructions} contains two constructions for regular maximal $(2,3)$-antichains of small regularity, and in the final Section~\ref{sec:conjectures} we propose two conjectures strengthening our results.     

\section{Lower bounds}\label{sec:lower_bounds}
In this section we start with a lower bound for the regularity of a maximal $(2,3)$-antichain which is valid for all values of $n$. The cases in which this lower bound is achieved are completely described, and we prove an asymptotic lower bound for large $n$.
\begin{lemma}\label{lem:bound_1}
Let $G$ be a graph with the property that every edge is contained in at least one triangle. Then for every edge $xy$,
\[t(x)\geqslant \frac{d(x)+t(xy)-1}{2}.\]
\end{lemma}
\begin{proof}
Fix an edge $xy$. There are $d(x)-t(xy)-1$ neighbours of $x$ that are not contained in a triangle with $x$ and $y$. To cover these there must be at least $(d(x)-t(xy)-1)/2$ triangles containing $x$, but not $y$. So the total number of triangles containing $x$ is at least
\[t(xy)+\frac{d(x)-t(xy)-1}{2}=\frac{d(x)+t(xy)-1}{2}.\qedhere\]
\end{proof}
\begin{theorem}\label{thm:16_bound}
If $G$ corresponds to an $r$-regular maximal $(2,3)$-antichain on $n$ points, then $\bar r\leqslant (n+3)/6$.
\end{theorem}
\begin{proof}
Fix a triangle $xyz$ and let $a_i$ denote the number of vertices outside $xyz$ that have exactly $i$ neighbours in $xyz$ ($i=0,1,2,3$). Then
\begin{align*}
a_0+a_1+a_2+a_3 &= n-3, \\
a_1+2a_2+3a_3 &= d(x)+d(y)+d(z)-6, \\
a_2+3a_3 &= t(xy)+t(yz)+t(xz)-3. 
\end{align*}
Subtracting the third from the second equation and using the first equation to bound $a_1+a_2$, we obtain
\[[d(x)-t(xy)]+[d(y)-t(yz)]+[d(z)-t(xz)]=a_1+a_2+3\leqslant n.\]
At least one of the three terms in square brackets on the left hand side, say $d(x)-t(xy)$, is at most $n/3$. We use Lemma \ref{lem:bound_1} to estimate $\bar r$,  
\[\bar r=d(x)-t(x)\leqslant d(x)-\frac{d(x)+t(xy)-1}{2}=\frac{d(x)-t(xy)+1}{2}\leqslant\frac{n+3}{6}.\qedhere\]  
\end{proof}
\begin{corollary}\label{cor:56bound}
  If there exists an $r$-regular maximal $(2,3)$-antichain on $n$ points, then $r\geqslant(5n-9)/6$.
\end{corollary}
Next we characterize the cases when equality in Corollary \ref{cor:56bound} can occur. 
\begin{corollary}\label{thm:quadrangles}
There exists a $(5n-9)/6$-regular maximal $(2,3)$-antichain if and only if $n\in\{3,9,15,27\}$. Furthermore, the extremal antichains for these values are unique up to isomorphism.  
\end{corollary}
\begin{proof}
We use the observation from \cite{GruettmuellerHartmannKalinowskiLeckRoberts2009} that the graphs on $n$ vertices having $n(n+3)/6$ edges and every edge contained in exactly one triangle are the triangle for $n=3$ and graphs corresponding to generalized quadrangles $GQ(2,(n-3)/6)$ which exist precisely for $n\in\{9,15,27\}$ (see \cite{payne2009finite}). So we just have to prove that in order to achieve $d(x)-t(x)=(n+3)/6$ we must have $n(n+3)/6$ edges and every edge has to be contained in exactly one triangle. In order to have equality in the proof of Theorem \ref{thm:16_bound} we need $d(x)-t(xy)=n/3$ for every edge $xy$. In particular, $n$ is divisible by $3$ and $t(xy)=d(x)-n/3$ is the same for all neighbours $y$ of $x$. If $d(x)\geqslant n/3+2$ then
\[t(x)=\frac12\sum_{u\,:\,xy\in E}t(xy)\geqslant d(x).\]
Hence we can conclude that $d(x)=n/3+1$ for all vertices $x$ and $t(xy)=1$ for all edges $xy$. 
\end{proof}

Now we prove an asymptotic upper bound on $\bar r$. Let $E$ and $T$ be the sets of edges and triangles in $G$, respectively. Intuitively, it seems reasonable to look for a graph with $d(x)=2\bar r$ and $t(x)=\bar r$ for every vertex $x$. The deviation from these values is measured by $\delta(x)=d(x)-2\bar r=t(x)-\bar r$ and we also put $\Delta=\sum_{x\in V}\delta(x)$. Summing over all vertices gives
\begin{align*}
  2\lvert E\rvert &=\sum_{x\in V}d(x)= 2\bar rn+\Delta, & 3\lvert T\rvert &=\sum_{x\in V}t(x)= \bar rn+\Delta.
\end{align*}
For the calculations it is convenient to define 
\begin{align*}
  \gamma&=\frac{\bar r}{n} & \beta&=\frac{1}{2n^2}\sum_{x\in V}\delta(x). 
\end{align*}
With this notation, 
\begin{align*}
  d(x) &= 2\gamma n+\delta(x), & t(x) &=\gamma n+\delta(x), \\
  \lvert E\rvert &= (\gamma+\beta)n^2, & \lvert T\rvert &= \frac{\gamma+2\beta}{3}n^2.
\end{align*}
\begin{lemma}\label{lem:gamma_beta} 
If $G$ corresponds to an $r$-regular maximal $(2,3)$-antichain then $\gamma\leqslant\beta+o(1)$.
\end{lemma}
\begin{proof} 
We may assume that $\lvert T\rvert=o(n^3)$, since otherwise $\bar r$ is negative for large $n$. The triangle removal lemma \cite{Fox2010,RuszaSzemeredi1976}, implies that $\lvert T\rvert \geqslant \lvert E\rvert/2+o(n^2)$, hence $(\gamma+2\beta)/3\geqslant(\gamma+\beta)/2+o(1)$, and rearranging terms gives the claimed inequality.
\end{proof}
\begin{lemma}\label{lem:t_xy_bound}
$\displaystyle\sum_{xy\in E}t(xy)^2\leqslant\frac12\sum_{x\in V}\delta(x)^2+o(n^3)$.
\end{lemma}
\begin{proof}
Fix a vertex $x$ and let $N(x)$ denote its neighbourhood. Then $\lvert N(x)\rvert=d(x)=2\gamma n+\delta(x)$ and
\[\sum_{y\in N(x)}t(xy)=2t(x)=2(\gamma n+\delta(x))=(2\gamma n+\delta(x)-1)\cdot 1+1\cdot(\delta(x)+1),\]
and, using $t(xy)\geqslant 1$ for all edges $xy$, this implies
\[\sum_{y\in N(x)}t(xy)^2\leqslant(\delta(x)+1)^2+(2\gamma n+\delta(x)-1)=\delta(x)^2+o(n^2),\]
and the claim follows by summation over all $x\in V$.
\end{proof}
\begin{theorem}\label{thm:asymp_1_10}
If $G$ corresponds to an $r$-regular maximal $(2,3)$-antichain, then $\bar r\leqslant (1/10+o(1))n$.
\end{theorem}
\begin{proof}
From the proof of Theorem \ref{thm:16_bound} we take the inequality 
\[d(x)+d(y)+d(z)-t(xy)-t(yz)-t(xz)\leqslant n\]
which is valid for every triangle $xyz\in T$. Summation over all triangles gives
\[\sum_{x\in V}t(x)d(x)-\sum_{xy\in E}t(xy)^2\leqslant n\lvert T\rvert=\frac{\gamma+2\beta}{3}n^3.\]
Substituting
\[t(x)d(x)=\left(\gamma n+\delta(x)\right)\left(2\gamma n+\delta(x)\right)=2\gamma^2 n^2+3\gamma n\delta(x)+\delta(x)^2\]
we obtain
\[\sum_{x\in V}\delta(x)^2-\sum_{xy\in E}t(xy)^2\leqslant\left[\frac{\gamma+2\beta}{3}-2\gamma^2-6\beta\gamma+o(1)\right]n^3,\]
and using Lemma \ref{lem:t_xy_bound} and the Cauchy-Schwarz inequality, 
\[\frac{4\beta^2n^4}{2n}\leqslant\frac12\sum_{x\in V}\delta(x)^2\leqslant\left[\frac{\gamma+2\beta}{3}-2\gamma^2-6\beta\gamma+o(1)\right]n^3,\]
hence
\[\frac{\gamma+2\beta}{3}-2\gamma^2-6\beta\gamma-2\beta^2\geqslant o(1).\]
Together with Lemma \ref{lem:gamma_beta}, this implies $\gamma\leqslant 1/10+o(1)$, and thus concludes the proof.
\end{proof}
\begin{corollary}\label{cor:asump_9_10}
  If there exists an $r$-regular maximal $(2,3)$-antichain on $n$ points, then $r\geqslant(9/10+o(1))n$.
\end{corollary}

\section{Constructions}\label{sec:constructions}
Here we present two different constructions for regular maximal $(2,3)$-antichains with small regularity. Section \ref{subsec:algebraic_construction} contains an algebraic construction for $n\equiv 0\pmod 3$. The regularity of these antichains is still $(1-o(1))n$, while in Section \ref{subsec:blowup_construction} we describe how certain two-coloured digraphs can be used to achieve regularity $cn$ with $c$ strictly less than $1$. 
\subsection{An algebraic construction}\label{subsec:algebraic_construction}
Let $n=3n'$, and let $\Gamma$ be an abelian group of order $\lvert\Gamma\rvert=n'$. Suppose that there are two sequences
\[A=(a_1,a_2,\ldots,a_p)\subseteq\Gamma\quad\text{and}\quad B=(b_1,\ldots,b_p)\subseteq\Gamma\]
satisfying the following conditions.
\begin{itemize}
\item The entries in both sequences are pairwise distinct, i.e., $a_i\neq a_j$ and $b_i\neq b_j$ for $i\neq j$.
\item The sequence $C=(b_1-a_1,b_2-a_2,\ldots,b_p-a_p)$ has pairwise distinct elements, i.e., $b_i-a_i\neq b_j-a_j$ for $i\neq j$.
\item $b_j-a_i\not\in C$ for all $i\neq j$.
\end{itemize}
Let $G$ be the graph with vertex set $V=\Gamma\times[3]$ and edge set
\begin{multline*}
E=\big\{\{(g,1),(h,2)\}\ :\ h-g\in A\big\}\cup\big\{\{(g,2),(h,3)\}\ :\ h-g\in C\big\}\\ \cup\big\{\{(g,1),(h,3)\}\ :\ h-g\in B\big\}.
\end{multline*}
The triangles in this graph are the sets of the form
\[\{(g,1),(g+a_s,2),(g+b_s,3)\}\]
for $g\in\Gamma$ and $s\in [p]$. Every vertex has degree $d(x)=2p$ and is contained in $t(x)=p$ triangles, so the corresponding maximal $(2,3)$-antichain is $(n-1-p)$-regular.

Let $p_{\max}(n')$ denote the maximal $p$ such that a group $\Gamma$ and sequences $A$ and $B$ satisfying the above conditions exist. Then our construction implies the following existence result.
\begin{theorem}
  If $n=3n'$, there exists a maximal, $r$-regular $(2,3)$-antichain for every $r\in\{n-1-p_{\max}(n'),\ldots,n-1\}$.
\end{theorem}

For $r\in\{n-1-p_{\max}(m'),\ldots,n-1\}$ the constructed antichain has $\tbinom{n}{2}-n(n-1-r)$ pairs and $n(n-1-r)/3$ triples. So the size is $\tbinom{n}{2}-\tfrac23n(n-1-r)$ and this is the smallest possible size for a $r$-regular maximal $(2,3)$-antichain which follows from the following lemma.
\begin{lemma}
 Let $G$ be a graph on $n$ vertices in which every edge is contained in at least one triangle and $d(x)-t(x)=n-1-r$ for all vertices $x$. Then $\lvert E\rvert-\lvert T\rvert\leqslant \tfrac23n(n-1-r)$.
\end{lemma}
\begin{proof}
We have 
\[2\lvert E\rvert-3\lvert T\rvert=\sum_{x\in V}d(x)-t(x)=n(n-1-r)\]
and, using $t(xy)\geqslant 1$ for every edge $xy$,
\[\frac32\lvert T\rvert = \frac12\sum_{xy\in E}t(xy)\geqslant \frac{\lvert E\rvert}{2}\]
Adding these two inequalities we get $\frac32\left(\lvert E\rvert-\lvert T\rvert\right)\geqslant n(n-1-r)$ which is the claim.
\end{proof}

As observed in \cite{RuszaSzemeredi1976}, sets without 3-term arithmetic progressions can be used to construct graphs with every edge contained in exactly one triangle. This is a special case of the above construction with $A=C$ and $B=2A$. Let $A$ be a subset of $\Gamma$ without nontrivial 3-term arithmetic progressions, i.e., the only solutions to $x+y=2z$ in $A$ are of the form $x=y=z=a$ for some $a\in A$. Then $B=2A=\{a+a\ :\ a\in A\}$ and $C=A$ satisfy our conditions, and the triangles in the graph are precisely the sets $\{(g,1),(g+a,2),(g+2a,3)\}$ for $g\in\Gamma$. So if $r_3(\Gamma)$ denotes the maximal size of a subset of $\Gamma$ without 3-APs we get $p_{\max}(n')\geqslant r_3(\Gamma)$. 
\begin{example}
For an integer $N$, let $r_3(N)$ denote the maximal size of a subset of $[N]$ without 3-term APs (in $\ints$). The best known lower bound for $r_3(N)$, based on a modification of a very old construction \cite{behrend1938sequences}, can be found in \cite{elkin2010improved} and \cite{green2010note}:  
\[r_3(N)\geqslant\frac{cN\log^{1/4}(N)}{2^{2\sqrt{2\log_2N}}}\]
for some positive constant $c$. Clearly, $r_3(\ints_{n'})\geqslant r_3(\lfloor n'/2\rfloor)$, and so we get that for integers $n\equiv 0\pmod 3$ there exists an $r$-regular maximal $(2,3)$-antichain $\mathcal A$ on $n$ points of size $\lvert\mathcal A\rvert=\tbinom{n}{2}-\tfrac23n(n-1-r)$ whenever $n>r\geqslant n-1-r_3(\lfloor n/6\rfloor)$. In particular, there is a positive constant $c$ such that  
\[r\geqslant n\left(1-\frac{c\log^{1/4}(n)}{2^{2\sqrt{2\log_2n}}}\right)\]
is sufficient for the existence of an $r$-regular maximal $(2,3)$-antichain on $n$ points.
\end{example}
\begin{example}
Let $A$ be the set of all integers $t$, $0\leqslant t<n'/2$ such that the ternary representation of $t$ contains only the digits $0$ and $1$, and interpret $A$ as a subset of $\ints_{n'}$. Then $A$ does not contain 3-term APs and this yields a regular maximal $(2,3)$-antichain with regularity
\[r\approx n-1-\lfloor n/6\rfloor^{\log_32}\leqslant n-cn^{0.63}.\]
\end{example}
\begin{example}
Let $\Gamma$ be the additive group of the vector space $\mathbb{F}_3^m$, i.e., $n=3^{m+1}$. The 3-AP free sets in this group are called \emph{cap sets} and there are constructions of size greater than $2.217^m$ for arbitrary large $m$ (see \cite{edel2004extensions}). This gives $r$-regular maximal antichains for 
\[r\approx n-1-n^{\frac{m\log 2.217}{(m+1)\log 3}}<n-1-n^{\frac{0.724m}{m+1}}.\] 
\end{example}

\subsection{A blowup construction}\label{subsec:blowup_construction}
Note that all of our constructions did not get a regularity below $(1-o(1))n$. We now describe a way to achieve regularity less than $cn$ for some positive $c$. We start with a special class of digraphs with a 2-colouring of the arcs. Let $D=(N,A_0\cup A_1)$ be a 2-coloured digraph with node set $N$ and arc set $A=A_0\cup A_1$ where the sets $A_0$ and $A_1$ are the colour classes. All our digraphs contain no loops and no parallel arcs. For such a digraph $D$ and any positive integer $t$ we can define a graph $G=G(D,t)$ on $2t\lvert N\rvert$ vertices as follows. To every node $v\in N$ correspond $2t$ vertices labeled by $(v,i,j)$ for $i\in\{1,2,\ldots,t\}$ and $j\in\{0,1\}$, and these $2t$ vertices induce a matching containing the edges $\{(v,i,0),(v,i,1)\}$ for $i\in\{1,2,\ldots,t\}$. Let $E_2$ denote the set of the $t\lvert N\rvert$ matching edges. Furthermore, let
\begin{align*}
  E_0 &=\left\{\{(u,i,j),\,(v,i',0)\}\ :\  (u,v)\in A_0,\ i,i'\in\{1,\ldots,t\},\ j\in\{0,1\}\right\},\qquad\text{and}\\
  E_1 &=\left\{\{(u,i,j),\,(v,i',1)\}\ :\  (u,v)\in A_1,\ i,i'\in\{1,\ldots,t\},\ j\in\{0,1\}\right\}.
\end{align*}
The edge set of the graph $G$ is then $E_0\cup E_1\cup E_2$. By construction, every edge of $G$ is contained in at least one triangle. Also every triangle contains at most one edge from the matching $E_2$. The next lemma provides a condition on the (undirected) triangles in the coloured digraph $D$ which ensures that every edge in $E_0\cup E_1$ is contained in precisely one triangle.
\begin{lemma}\label{lem:unique}
If all triangles in $D=(N,A_0\cup A_1)$ have the form $\{u,v,w\}$ with $(u,v),\,(v,w)\in A_0$ and $(u,w)\in A_1$ or $(u,v),\,(v,w)\in A_1$ and $(u,w)\in A_0$, then in $G=G(D,t)$ every edge $\{x,y\}\in E_0\cup E_1$ is contained in a unique triangle.
\end{lemma}
\begin{proof}
An edge $e=\{(u,i,j),\,(v,i',j')\}$ with $u\neq v$ is contained in precisely one triangle corresponding to the arc $(u,v)\in A_{j'}$ or to the arc $(v,u)\in A_{j}$. Any other triangle containing the edge $e$ comes from a triangle $\{u,v,w\}$ in the digraph $D$. But if all triangles in $D$ have the form required in the hypothesis of the lemma then the structure of the corresponding subgraphs in $G$ is shown in Figure \ref{fig:good_triangle} and it does not contain any unwanted triangles.   
\begin{figure}[htb]
  \centering
\includegraphics[width=.45\textwidth]{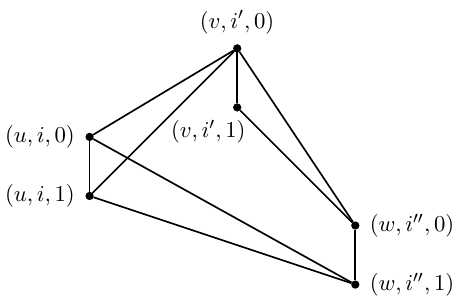}
  \caption{Subgraph from a triangle $\{u,v,w\}$ with $(u,v),\,(v,w)\in A_0$ and $(u,w)\in A_1$.}
  \label{fig:good_triangle}
\end{figure} 
\end{proof}
Under the conditions of Lemma \ref{lem:unique} it is easy to determine $d(x)$ and $t(x)$ for the vertices of the graph $G$. Let $d_i^-(v)$ and $d_i^+(v)$ denote the number of arcs in $A_i$ entering and leaving node $v$, respectively. Then we have
\begin{align*}
  d(x) &=
  \begin{cases}
    1+t\left(d_0^+(u)+d_1^+(u)+2d_0^-(u)\right) & \text{if }x=(u,i,0)\text{ for some }i,\\
    1+t\left(d_0^+(u)+d_1^+(u)+2d_1^-(u)\right) & \text{if }x=(u,i,1)\text{ for some }i,
  \end{cases}\\
  t(x) &=
  \begin{cases}
    t\left(d_0^+(u)+d_1^+(u)+d_0^-(u)\right) & \text{if }x=(u,i,0)\text{ for some }i,\\
    t\left(d_0^+(u)+d_1^+(u)+d_1^-(u)\right) & \text{if }x=(u,i,1)\text{ for some }i.
  \end{cases}
\end{align*}
In order to make $d(x)-t(x)$ constant we have to require that the in-degrees in both colours are the same constant over the node set $N$, i.e., $d_0^-(u)=d_1^-(u)=d$ for some $d$ and all $u\in N$. In the graph $G=G(D,t)$ with $n=2t\lvert N\rvert$ vertices we have $d(x)-t(x)=1+td=1+dn/2\lvert N\rvert$, and the regularity of the corresponding maximal $(2,3)$-antichain is
\[r=\left(1-\frac{d}{2\lvert N\rvert}\right)n-2.\]   
Figure \ref{fig:digraphs} shows two digraphs on $7$ and $8$ nodes, respectively, satisfying the condition in Lemma \ref{lem:unique}, and having constant in-degree $1$ in both colours. 
\begin{figure}[htb]
\begin{minipage}{.48\linewidth}
\begin{center}
\includegraphics[width=.5\textwidth]{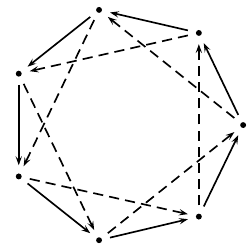} 
\end{center}
\end{minipage}\hfill
\begin{minipage}{.48\linewidth}
\centering
\includegraphics[width=.75\textwidth]{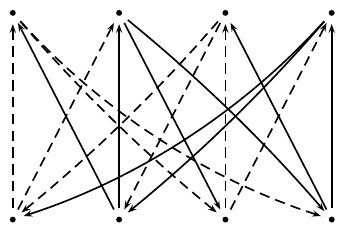}
\end{minipage} 
\caption{Two arc-coloured digraphs (colours indicated by solid and dashed lines).}\label{fig:digraphs}
\end{figure}
For the corresponding graphs $G=G(D,t)$ we obtain $d(x)=1+4t$ and $t(x)=3t$ for all vertices $x$, and this implies the following theorem.
\begin{theorem}\label{thm:small_regularity}
  For every $n\equiv 0\pmod{14}$ there is a $(13n/14-2)$-regular maximal $(2,3)$-antichain on $n$ points, and for every $n\equiv 0\pmod{16}$ there is a $(15n/16-2)$-regular maximal $(2,3)$-antichain.  
\end{theorem}

\section{Open problems}\label{sec:conjectures}
We conjecture that the first part of Theorem~\ref{thm:small_regularity} gives asymptotically the smallest possible regularity.
\begin{conjecture}\label{conj:main_conjecture}
The smallest possible regularity of a maximal $(2,3)$-antichain on $n$ points is $(13/14+o(1))n$.
\end{conjecture}
In the proof of the $9/10$-bound in Corollary~\ref{cor:asump_9_10} we assumed the worst case that for every triangle $xyz$ every vertex outside this triangle is adjacent to either one or two of the triangle vertices. It is not difficult to see that the case where an outside vertex is adjacent to all three vertices of a triangle can be neglected for the asymptotics, because the number of 4-cliques is $o(n^3)$. But the factor $9/10$ can be improved using any cubic lower bound on the number of copies of the graph $H$ on four vertices consisting of a triangle and an isolated vertex. Here are some observations that might be used to force copies of $H$.
\begin{itemize}
\item The condition that $d(x)-t(x)$ is large, say $cn$ for some $c\geqslant 1/14$, implies that the neighbourhood of every vertex has many connected components: at least $cn$ (and exactly $cn$ if the neighbourhood induces a forest).
\item So every neighbourhood contains an induced matching of size $cn$.
\item The union of $n$ induced matchings is small, only $o(n^2)$ edges~\cite{RuszaSzemeredi1976}.
\end{itemize}
As a first step towards Conjecture~\ref{conj:main_conjecture} one might try to prove that $13n/14$ is best possible for the type of construction described in Section \ref{subsec:blowup_construction}.
\begin{conjecture}\label{conj:secondary_conjecture}
 Let $D=(N,A_0\cup A_1)$ be an arc-coloured digraph with the following properties.
 \begin{enumerate}
 \item All triangles have the form $\{u,v,w\}$ with $(u,v),(v,w)\in A_0$ and $(u,w)\in A_1$ or $(u,v),(v,w)\in A_1$ and $(u,w)\in A_0$.
 \item Every node has exactly $d$ incoming arcs in each of the sets $A_0$ and $A_1$.  
 \end{enumerate}
Then $d\leqslant \lvert N\rvert/7$.
\end{conjecture}

%\bibliographystyle{plain}
%\bibliography{/home/tk776/Dropbox/bib/Antichains}

\end{document}